\documentclass[12pt]{amsart}
\usepackage[all]{xy}
\usepackage{graphicx}
\usepackage{amssymb}
\newenvironment{lemma}{\smallskip\begin{trivlist}
 \item[\hspace{\labelsep}{\noindent\bf Lemma.}]\it
 }{\end{trivlist}\smallskip}
\newenvironment{propo}{\smallskip\begin{trivlist}
 \item[\hspace{\labelsep}{\noindent\bf Proposition.}]\it
 }{\end{trivlist}\smallskip}

\newenvironment{theor}{\smallskip\begin{trivlist}
 \item[\hspace{\labelsep}{\noindent\bf Theorem.}]\it
 }{\end{trivlist}\smallskip}

\newenvironment{coro}{\smallskip\begin{trivlist}
 \item[\hspace{\labelsep}{\noindent\bf Corollary.}]\it
 }{\end{trivlist}\smallskip}

\newcommand{\aA}{\lower1pt\hbox{$\!_A$}}

\newcommand{\aAm}{\lower1pt\hbox{$\!_{A_m}$}}

\newcommand{\A}{\mathbb{A}}

\newcommand{\D}{\mathbb{D}}
\newcommand{\N}{\mathbb{N}}
\newcommand{\R}{\mathbb{R}}

\newcommand{\Z}{\mathbb{Z}}

\renewcommand{\dim}{\hbox{\rm dim}}
\newcommand{\udim}{\hbox{\bf dim}}
\newcommand{\Ext}{\hbox{\rm Ext}} 
\newcommand{\End}{\hbox{\rm End}}
\newcommand{\Hom}{\hbox{\rm Hom}}

\newcommand{\ind}{\hbox{\rm ind}}

\renewcommand{\mod}{\hbox{\rm mod}}

\newcommand{\supp}{\hbox{\rm supp}\,}

\def\raya{\raise1.5pt\hbox to 25pt{\vrule height1.5pt depth-1pt
           width25pt}}
\newcommand{\bulito}{\ {\scriptstyle \bullet}\, \ }
\def\subsetnoteq{\mathbin{\hbox{$\subseteq \joinrel \hskip-8pt \lower3pt
                 \hbox{$\scriptscriptstyle /$}\ $}}}
\def\aderecha#1{\smash{\mathop{\longrightarrow}\limits^{#1}}}

\def\mapdowndown#1{\Bigg\downarrow \hskip-7pt\lower10pt\hbox{$\downarrow$}
        \rlap{$\vcenter{\hbox{$\scriptstyle#1$}}$}}
\def\vmapdown#1{\raise5pt\hbox{$\sssty\vert$} 
         \hskip-2.25pt \raise17pt\hbox{$\sssty\vert$} 
         \hskip-4.25pt \lower13pt\hbox{$\downarrow$}
         \rlap{$\vcenter{\hbox{$\scriptstyle#1$}}$}}
\def\sssty{\scriptscriptstyle}

\begin{document}
\baselineskip=18pt plus 1 pt

\title[Algebras accepting a maximal root]{Algebras whose Tits form accepts a maximal omnipresent root}
\author[Jos\'e A. de la Pe\~na and Andrzej Skowro\'nski
]{Jos\'e A. de la Pe\~na$^1$ and Andrzej Skowro\'nski$^2$}
\maketitle
\noindent
{\footnotesize $^1$ Instituto de Matem\'aticas, Universidad Nacional Aut\'onoma de M\'exico, M\'exico 04510 D.F., M\'exico. 
and Centro de Investigaci\'on en Matem\'aticas AC, La Valenciana, Guanajuato 36240 Gto. M\'exico
e-mail: jap@matem.unam.mx and jap@cimat.mx\\
$^2$ Faculty for Mathematics and Computer Science, Nicolaus Copernicus University, Chopina 12/18, 87-100 Toru{\'n}, Poland. e-mail: skowron@mat.uni.torun.pl}

\vskip.5cm
{\footnotesize {\bf Abstract:} Let $k$ be an algebraically closed field and $A$ be a finite-dimensional associative basic $k$-algebra of the form $A=kQ/I$ where $Q$ is a quiver without oriented cycles or double arrows and $I$ is an admissible ideal of $kQ$. We consider roots of the Tits form $q_A$, in particular in case $q_A$ is weakly non-negative. We prove that for any maximal omnipresent root $v$ of $q_A$, there exists an indecomposable $A$-module $X$ such that v={\bf dim} X. Moreover, if $A$ is strongly simply connected, the existence of a maximal omnipresent root of $q_A$ implies that $A$ is tame of tilted type.}

{\footnotesize {\bf Key words}: root of the Tits form, weakly non-negative quadratic form, maximal omnipresent root, tilted algebra, tame algebra, exceptional index. }
\medskip

Let $k$ be an algebraically closed field and $A$ be a finite-dimensional associative basic $k$-algebra with unity. Write $A=kQ/I$ where $Q$ is a finite quiver and $I$ is an admissible ideal of the
path algebra $kQ$, see \cite{1}, \cite{14}. In this work, an $A$-module will always be a finitely generated left $A$-module.  

A fundamental problem in the {\em representation theory of
algebras\/} is the classification of indecomposable $A$-modules
(up to isomorphism).  One of the early successes of modern representation
theory was the identification by Gabriel \cite{13} of the one-to-one correspondence between isoclasses of indecomposable modules and roots of the quadratic form 
$q_Q(x_1,\ldots,x_n)= \sum\limits_{i=1}^n x_i^2 -\sum\limits_{i \to j} x_i x_j$ for the quivers $Q$ whose underlying graph is a Dynkin diagram. Those quivers $Q$ correspond to the representation-finite hereditary algebras $A=kQ$.

Recall that an algebra $A$ is {\em representation-finite} if there are only finitely many indecomposable $A$-modules up to isomorphism. More generally, an algebra $A$ is {\em tame\/} if, for every positive natural number $n$, almost every indecomposable $n$-dimensional $A$-module 
is isomorphic to a module belonging to a finite number of $1$-parametric families of modules. 
A central result of Drozd \cite{12} states that every finite-dimensional $k$-algebra is either tame or wild, the later case meaning that the classification of the indecomposable $A$-modules implies the classification of the indecomposable modules over any other finite-dimensional algebra. 

In general, we shall consider algebras $A=kQ/I$ which are {\em triangular\/}, that is,
$Q$ has no oriented cycles. Following Bongartz \cite{3}, for a triangular algebra $A$, the {\em Tits
quadratic form\/} $q_A\colon \Z^{Q_0}\to \Z$ is defined by
$$q_A(v)=\sum_{i\in Q_0}v(i)^2-\sum_{(i\to j)\in Q_1}v(i)v(j)
+\sum_{i,j\in Q_0}r(i,j)v(i)v(j)$$
where $Q_0$ (resp. $Q_1$) denotes the set of vertices (resp. arrows)
of $Q$ and $r(i,j)$ is the number of elements in $R\cap I(i,j)$
whenever $R$ is a minimal set of generators of $I$ contained in
$\bigcup\limits_{i,j\in Q_0}I(i,j)$. For a representation-finite algebra $A$, it was shown in \cite{3}, that $q_A$ is {\em weakly positive\/}, that is, $q_A(v)>0$ for any vector $0\ne v\in \N^{Q_0}$, and, in case the Auslander-Reiten quiver of $A$ accepts a preprojective component, then there is a one-to-one correspondence $X \to \udim X$ between the isoclasses of indecomposable $A$-modules and the (finitely many) roots of the Tits form $q_A$. For a tame algebra, the first named author showed \cite{25}, that $q_A$ is weakly non-negative, that is, $q_A(v)\ge 0$ for any $v\in \N^{Q_0}$. Although the converse is not true, there are important cases where it holds. For instance, if $A$ is a tilted algebra, Kerner's analysis \cite{20} shows that $A$ is tame exactly when $q_A$ is weakly non-negative, and in that case the dimension vector function yields an injection of the isoclasses of indecomposable directing $A$-modules into the roots of the Tits form $q_A$. Recently, it was shown by Br\"ustle and the authors of this work that a strongly simply connected algebra $A$ is tame if and only if $q_A$ is weakly non-negative, \cite{7}.

In this work we extend some of the above results in the following way. Given a unit form $q(x_1,\ldots,x_n)= \sum\limits_{i=1}^n x_i^2 + \sum_{i \ne j} q_{ij}x_ix_j$, a root $v \in \Z^n$ (that is, $v$ has non-negative coordinates and $q(v)=1$) is {\em omnipresent} if $v(i) > 0$ for every $i=1,\ldots,n$; it is {\em maximal} if any root $w \ge v$ is $w=v$. We denote by $q(-,?)$ the associated symmetric bilinear form. We say that the form $q$ is {\em slender} if $q_{ij} \ge -1$ for every pair $i \ne j$. The first result is fundamental in our context. 

\vskip.5truecm

{\bf Theorem 0.} {\em Let $A$ be a triangular algebra such that $q_A$ is a slender form with a maximal omnipresent root $v$. Then $q_A$ is weakly non-negative. Moreover, $v(i) \le 12$ for every $i \in Q_0$.} 

\vskip.5truecm
The result follows directly from \cite{18} and the corresponding result for the coordinates of maximal omnipresent roots of weakly non-negative unit forms proved in \cite{11}. We review the proof in Section 1. We do not know what is the optimal bound for the coordinates of a maximal omnipresent root of the Tits form of an algebra, we conjecture it is $6$.

The problem of {\em realization} of a root $v$ as the dimension vector $v= \udim X$ of an indecomposable module $X$ is difficult. We shall discuss the problem of realization of maximal omnipresent roots for a class of algebras which includes the schurian tame algebras. 

\vskip.5truecm
{\bf Theorem 1.} {\em Let $A=kQ/I$ be a triangular algebra without double arrows such that $q_A$ accepts a maximal omnipresent root $v$. Then there exists an indecomposable module $X$ such that $v= \udim X$.}

The main argument for the realization of a root $v$ as the dimension vector $v= \udim X$ of an indecomposable module $X$ requires the consideration of algebraic varieties of modules. Indeed, consider the variety of modules $\mod_A(v)$ with the Zariski topology equipped with the action of the algebraic group $G(v)= \prod\limits_{i \in Q_0} GL_k(v(i))$. For every irreducible component $C$ 
of $mod_A (v)$ we get a {\it generic decomposition}  $v=w_1 +\cdots +w_s$ in $C$, for which
$$\{X\in C\colon X=X_1 \oplus \cdots \oplus X_s\ \hbox{with}\
\udim \ X_i=w_i\ \hbox{and}\ X_i\ \hbox{indecomposable}\}$$
\noindent 
contains an open and dense subset of $C$. We prove that under the hypothesis of the Theorem, $s=1$. 

Recall from \cite{34} that a triangular algebra $A$ is {\em strongly simply connected\/} if for every convex subcategory $B$ of $A$ the first Hochschild cohomology $H^1(B)$ vanishes. The relevance of strongly simply connected algebras is associated to the construction of Galois coverings $R \to R/G=A$, for certain finite dimensional algebras $A$, with  strongly simply connected categories $R$ and torsion-free groups $G$ of automorphisms of $R$. The main result of \cite{7} is used to prove the next Theorem.

\vskip.5cm
{\bf Theorem 2.}
{\em Let $A=kQ/I$ be a strongly simply connected algebra such that $q_A$ accepts a maximal omnipresent root $v$. Then $A$ is a tame tilted algebra and there is a one-to-one correspondence $X \to \udim X$ between the isoclasses of indecomposable sincere $A$-modules and the omnipresent roots of $q_A$.}

Following Ringel \cite{31}, we define {\em exceptional indices} for maximal omnipresent roots $v$ of  weakly non-negative unit forms $q(x_1,\ldots,x_n)$. Indeed, we have $q(v,e_i) \ge 0$ for all $i$ and $q(v,e_j) > 0$ for at most $2$ indices $j$. As a consequence of Theorem 1 we shall prove the following generalization of the main Theorem of \cite{22}.

\vskip.5truecm
{\bf Theorem 3.} {\em Let $A$ be a strongly simply connected algebra such that $q_A$ accepts a maximal omnipresent root $v$ with two exceptional indices. 
Then there is a unique indecomposable sincere $A$-module.}

For concepts not explicitely introduced in the paper as well as general background we refer the reader to \cite{2}, \cite{31}, \cite{32} and \cite{33}.

The work for this paper started during the visit of the second named author to the Instituto de Matem\'aticas at UNAM during the spring of 2009.
Both authors acknowledge support from the Consejo Nacional de Ciencia y Technologia of Mexico.
The second author has also been supported by the Research Grant No. N N201 269135 of the Polish Ministry of Science 
and Higher Education.

\section{Unit forms and algebras.}

\subsection{}
We recall some elementary facts of linear algebra. Let $M = (a_{ij})$ be a symmetric integral $n \times n$-matrix such that $a_{ii}=2$ for $i=1,\ldots,n$.  The corresponding bilinear form $q(x,y)=xMy^t$, for any couple of row vectors $x,y$, satisfies $q(x+y)=q(x) + q(x,y) +q(y)$, where $q$ is the quadratic form associated to
$M$, that is, $q(x) = {1 \over 2} x M x^{t}$. The form $q$ is a {\em unit form}, that is,
it has the shape $q(x_1,\ldots,x_n)=\sum\limits^n_{i=1}x^2_i+\sum\limits_{i<j}q_{ij}x_ix_j$ for some integral numbers $q_{ij}$. Denote by $e_1, \ldots,e_n$ the canonical basis of $\Z^n$, then we get $q(e_i,e_j)=q_{ij}$.

Given an index $i=1, \ldots,n$, the {\em reflection} of $q$ at $i$ is the function $\sigma_{i}(z)=z - q(z,e_i)e_i$ on $\Z^n$. Observe that for a root $v$ of $q$, the reflection $\sigma_i(v)$ is again a root.

\begin{lemma}
Let $q$ be a weakly positive unit form. 
Then the following statements hold.

(a) For every pair $i,j$ such that $q_{ij} < 0$, we have $q_{ij} = -1$.

(b) For any root $v$ of $q$ we have $-1 \le q(v,e_i) \le 1$, for every $1 \le i \le n$ such that $v(i) \ne 0$. In particular, there is a chain of roots $e_{j(0)}=v_1, v_2,\ldots, v_m=v$ such that $v_{i-1}= v_i -e_{j(i)}$ for $2 \le i \le m$ and some indices $j(i)$.

(c) Drozd \cite{Dr0}: $q$ accepts only finitely many positive roots.
\end{lemma}

\begin{proof} 
(a) Since $0 < q(e_i+e_j)= 2+ q_{ij}$, hence $-2 < q_{ij}$.

(b) Let $v$ be a root and $1 \le i \le n$ with $v(i) \ne 0$, then both $v-e_i$ and $v+e_i$ are non-negative vectors. Hence $0 \le q(v-e_i)=2-q(v,e_i)$ (it is $0$ only if $v=e_i$) and $0< q(v+e_i)=2 +q(v,e_i)$ yield the desired inequalities. Since $2=q(v,v)= \sum\limits_{i=1}^n v(i) q(v,e_i)$, there is some index $j$ with $q(v,e_j)=1$ and $v(j) \ne 0$. Hence $\sigma_j(v)$ is a root and we repeat inductively the argument.

(c) Consider $q$ as a function $q : \R^n \to \R$. By
continuity $q (z) \geq 0$ in the positive cone $K = (\R^{n})^{+}$. By
induction on $n$, it can be shown that $q (z) > 0$ for any $0 \not= z
\in K$. Let $0 < \gamma$ be the minimal value reached by $q$ on $\{z
\in K : \Vert z\Vert = 1\}$ (a compact set). Then a positive
root $z$ of $q$ satisfies $\gamma \leq q \left({z \over \Vert z
\Vert} \right) = {1 \over \Vert z \Vert^{2}}$, that is,
$\Vert z \Vert \leq \sqrt{1 / \gamma}$.
\end{proof}

\subsection{} 

The following was essentially proved in \cite{22}.

\begin{lemma}
Let $q\colon \Z^n \to \Z$ be a weakly non-negative unit form. Then the following statements hold.

(a) For every pair $i,j$ such that $q_{ij} < 0$, we have $-2 \leq q_{ij}$.

(b) For an omnipresent root $v$ of $q$ we have $-2 \le q(v,e_i) \le 2$, for every $1 \le i \le n$. 
If $v$ is a maximal root of $q$ then $0 \le q(v,e_i)$, for every $1 \le i \le n$.

(c) If $q$ accepts a maximal omnipresent root, then $q$ has only finitely many omnipresent roots.
\end{lemma}
\begin{proof}
(a) As in (1.1). For (b), consider an omnipresent root $v$ of $q$ and $1 \le i \le n$. It is clear that $-2 \le q(v,e_i) \le 2$ and if $q(v,e_i) < 0$ then $v \le \sigma_i(v)$, is a root bigger than $v$ in the partial order given by the numerical coordinates. 

(c) Let $v$ be a maximal omnipresent root of $q$ and assume that $q$ accepts infinitely many omnipresent roots. Then there is an ascending chain of omnipresent roots $(v_i)_{i \in \N}$ in the considered partial order.  By (a), there should be two roots $v_r < v_s$ such that $q(v_r,e_i)=q(v_s,e_i)$, for every $1 \le i \le n$. Then the non-zero, non-negative vector $w=v_s-v_r$ satisfies $q(w, e_i)=0$ for every $1 \le i \le n$. Hence $v+w$ is strictly bigger than $v$ and $q(v+w)=1+q(v,w)= 1$, which contradicts the maximality of $v$.
\end{proof}

\subsection{} 
We show the weak non-negative version of (1.1).

\begin{propo}
Let $q: \Z^n \to \Z$ be a weakly non-negative unit form. Let $v$ be a positive root of $q$.
Then there are a vertex $j$ and a positive root $y$ of $q$ such that there exists a sequence of reflections $\sigma_{i_1},\ldots,\sigma_{i_s}$ satisfying:

(1) $\sigma_{i_s}\ldots \sigma_{i_1}(e_j)=y \le v$;

(2) $\sigma_{i_t}\ldots
\sigma_{i_1}(e_j)= e_j +e_{i_1}+\cdots +e_{i_t}$ for all $1 \le t \le s$;

(3) every vector $0 \le u \le y$ has $q(u) \ge 1$;

(4) $q(v-y)=0$.
\end{propo}
\begin{proof}
We shall prove the result by induction on $|v|=\sum\limits_{i}v(i)$. The case $|v|=1$ is obvious.

Assume first that there is a vector $0 \le u \le v$ with $q(u)=0$. Consider the vector $0 \le v'=v-u$
which satisfies $0 \le q(v')=1-q(v,u) = 1- \sum \limits_{i}v(i) q(u,e_i) \le 1$, that is, one of two situations occur:

(a) $q(u,e_i)=0$ for every $i \in {\rm supp \,}v$ and $q(v')=1$;

(b) there is a unique $i$ with $v(i) \ne 0$ and $q(u,e_i)>0$. Then $v(i)=1=q(u,e_i)$ and $q(v')=0$.

In the former case, since $|v'| < |v|$, by induction hypothesis, there is a vector $y$ satisfying (1) to (3) and (4'): $q(v'-y)=0$. Therefore $q(v-y)=q(v'-y+u)=q(v'-y)=0$ and $y$ satisfies (1) to (4).
In the later case, $q(v'-e_i)=1$ and there is a vector $y \le v'-e_i$ satisfying (1) to (3) and (4''): $q(v'-e_i-y)=0$. Therefore $1=q(y)+q(v',y)=q(v'-y)=q(v'-y,e_i)+1$, due to the fact that $q(v',e_j)=0$ for every $v'(j) \ne 0$, and  $q(v-y)=q((v'-e_i-y)+u+e_i)=2+q(v'-e_i-y,e_i)= q(v'-y,e_i)=0$. Hence $y$ satisfies (1) to (4).

Without lost of generality, we assume that $v$ satisfies (3). Since $2=q(v,v)= \sum\limits_{i \in Q_0}v(i) q(v,e_i)$ then there are vertices $i$ with $v(i) \ne 0$ and  
$1 \le q(v,e_i) \le 2$, the last inequality due to (1.2). If $q(v,e_j)=2$ then $q(v-e_j)=0$ and we  are done. Assume $q(v,e_i) \le 1$ for all $i \in {\rm supp \,} v$ and choose $i_0$ with $v(i_0)>0$ and $q(v,e_{i_0})=1$ and consider the positive root $v_1=v-e_{i_0}$. Continue choosing vertices $i_1,\ldots,i_s$ such that $v_t=v_{t-1}-e_{i_{t-1}}$ is a positive root, for $1 \le t \le s$, with $v=v_0$ and every vector $0 \le u \le v$ satisfying $q(u) \ge 1$. Assume that $s$ is maximal. Then $y'=\sigma_{i_s}\ldots \sigma_{i_1}(e_{i_0})$ satisfies (1), (2) and (3). Moreover, either $v=y'$ or there exists a vertex $j$ such that $(v-y')(j) >0$ and $q(v-y',e_j)=2$. 
Assume that there exists a vertex $j$ such that $(v-y')(j) >0$ and $q(v-y',e_j)=2$. Since $q(v,e_j) \le 1$ then $q(v,e_j)=1= -q(y',e_j)$. Therefore $y=\sigma_j(y')=y'+e_j$ satisfies (1) to (4).
\end{proof}

\subsection{} 
Let $v$ be a maximal root of a weakly non-negative unit form $q$. If for some vertex $i$ we had $q(v,e_i)<0$ then $\sigma_i(v)=v+q(v,e_i)e_i$ would be a root bigger than $v$. We shall say that a root $v$ of $q$ is {\em locally maximal} if $q(v,e_i) \ge 0$ for every index $i$. Maximal roots are locally maximal, but the converse does not hold as the following example shows.
Let $A$ be the algebra given by the quiver with relations below

$${\footnotesize 
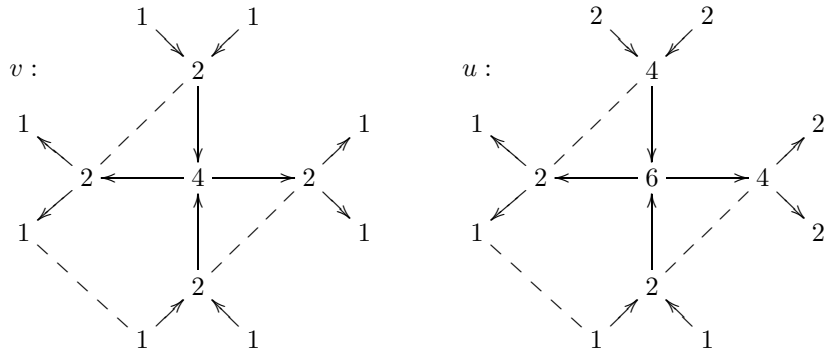
\xymatrix@C10pt@R8pt{&&1\ar[rd]&&1\ar[ld]&&\\v:&&&2 \ar[dd]\ar@{--}[lldd]&&&\\
1 &&&&&&1\\&2\ar[ld]\ar[lu]&&4 \ar[rr]\ar[ll]&&2 \ar[rd]\ar[ru]&\\1\ar@{--}[rrdd]&&&&&&1\\&&&2 \ar[uu]\ar@{--}[rruu]&&&\\&&1\ar[ru]&&1 \ar[lu]&&
}
\hskip1cm \xymatrix@C10pt@R8pt{&&2\ar[rd]&&2\ar[ld]&&\\u:&&&4 \ar[dd]\ar@{--}[lldd]&&&\\
1 &&&&&&2\\&2\ar[ld]\ar[lu]&&6 \ar[rr]\ar[ll]&&4 \ar[rd]\ar[ru]&\\1\ar@{--}[rrdd]&&&&&&2\\&&&2 \ar[uu]\ar@{--}[rruu]&&&\\&&1\ar[ru]&&1 \ar[lu]&&
}
}$$
\noindent
Consider the Tits form $q_A$ of the strongly simply connected tame algebra $A$, the form $q_A$ is therefore weakly non-negative. The vector $v$ in the example is a locally maximal omnipresent root but the vector $u$ is also a root of $q_A$.
The following remark proves, in our context, the existence of exceptional indices.

\begin{propo}
Let $q\colon \Z^n \to \Z$ be a weakly non-negative unit form with a locally maximal omnipresent root $v$. Then one of the following three situations occurs:

(i) There exists a unique index $i$ with $q(v,e_i) > 0$, for all other $j \ne i$ we have $q(v,e_j)=0$. Moreover, $q(v,e_i)=1$ and $v(i)=2$.

(ii) There are two indices $a \ne b$ with $q(v,e_i) > 0$ for $i=a,b$, for all other $j \ne a,b$ we have $q(v,e_j)=0$. Moreover, $q(v,e_a)=1=q(v,e_b)$ and $v(a)=1=v(b)$.

(iii) There is a unique index $j$ with $q(v,e_j)>0$, for all other $i \ne j$ we have $q(v,e_i)=0$. Moreover, $q(v,e_i)=2$ and $v(j)=1$.

In case $v$ is maximal only (i) or (ii) may happen. Moreover, in that case, if $E$ is the set of exceptional indices, the quadratic form $q'=q^{(E)}$ obtained as the restriction of $q$ to the vertices not in $E$ is weakly positive.
\end{propo}
\begin{proof}
Consider $v$ a locally maximal omnipresent root of $q$. Then $2 = q(v,v) = \sum\limits_{i=1}^n v(i) q(v,e_i)$ and each $q(v,e_i) \ge 0$. 

In case $v$ is maximal and $q(v,e_i)\ge 2$ we get that $q(2v-e_i)=4 -2 q(v,e_i) +1 \le 1$. Since $2v-e_i$ is omnipresent, then the hypothesis yield $q(2v-e_i)=0$, then $q(3v-e_i)= 1 + q(v,2v-e_i)=1$, a contradiction to the maximality of $v$. 

To show that $q'$ is weakly positive, assume that $q(w)=0$ for some vector $0 \le w \in \Z^n$ with $w(i)=0$ for $i \in E$. Since $q$ is weakly non-negative $q(w,e_j)\ge 0$ for all $j$ such that $w(j) \ne 0$. Then $0= \sum\limits_{j=1}^n w(j) q(v,e_j)=q(v,w)= \sum\limits_{i \in E} v(i) q(w,e_i)$ and therefore $q(w,e_i)=0$ for all $i \in E$. But then $q(v+w)=1 +q(v,w)=1$, contradicting the maximality of $v$. Thus $q'$ is weakly positive.
\end{proof} 

\begin{coro}
Let $q\colon \Z^n \to \Z$ be a weakly positive unit form and $v$ a root of $q$. Then $v$ is locally maximal if and only if it is maximal.
\end{coro}
\begin{proof}
Assume $v$ is locally maximal but not maximal, say $u \ne v \le u$ for a root $u$. Then $u-v$ is a non-negative non-trivial vector satisfying $q(u-v)=2- \sum\limits_{i} u(i)q(v,e_i) \le 1$. Since $q$ is weakly positive then $q(u-v)=1$ and there is a unique index $j$ with $q(v,e_j) \ge 1$. In fact $q(v,e_j)=1$ and $2=v(j) \le u(j) \le 1$ yields a contradiction.
\end{proof}

\subsection{}
{\em Examples.} We consider unit forms $q$ defined by diagrams, where as usual, $s$ full edges (resp. $s$ dotted edges) between the vertices $i$ and $j$ means that $q_{ij}=-s$ (resp. $q_{ij}=s$).

(1) Weakly non-negative unit forms with a maximal omnipresent root $v$ have $v(i) \le 12$ for every index $i$, and the bound is optimal as shown in \cite{11}. The following is an example of a unit form $q$ with a maximal omnipresent root $v$ with $v(j)=12$ for some vertex $j$ and marked exceptional indices:
$$\footnotesize
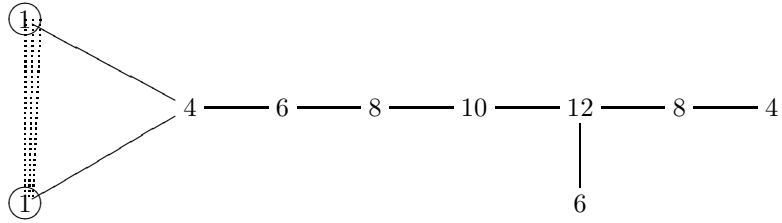
{\xymatrix{ & & \save[]+<0mm,0mm>*\frm<6pt>{o}*\txt<3pc>{${\rm 1}$}\restore \save[]+<-10.5mm,0mm>*\txt<2pc>{$\ar@{.}[ldd]$}\restore \save[]+<-10mm,0mm>*\txt<2pc>{$\ar@{.}[ldd]$}\restore \save[]+<-9.5mm,0mm>*\txt<2pc>{$\ar@{.}[ldd]$}\restore& & &&&&&&\\
            & & &&{\rm 4}\ar@{-}[llu] \ar@{-}[lld]\ar@{-}[r]&6\ar@{-}[r]&8\ar@{-}[r]&10\ar@{-}[r]&12\ar@{-}[d]\ar@{-}[r]&8\ar@{-}[r]&4\\
             &   &\save[]+<0mm,0mm>*\frm<6pt>{o}*\txt<3pc>{${\rm 1}$}\restore  &&&&&&6& & \\
& & &
}}$$

(2) Consider the forms associated to the diagrams (1) and (2) below:

$$
\footnotesize{\xymatrix{ {\rm 1} \ar@{-}[rd]&& &&& {\rm 1} \ar@{-}[ld]\\
{\rm 1} \ar@{-}[r]& {\rm 2} \ar@{-}[r] & \save[]+<0mm,0mm>*\frm<6pt>{o}*\txt<3pc>{${\rm 1}$}\restore \ar@{.}[r]& \save[]+<0mm,0mm>*\frm<6pt>{o}*\txt<3pc>{${\rm 1}$}\restore \ar@{-}[r] &{\rm 2} \ar@{-}[r]& {\rm 1} \\
                         {\rm 1} \ar@{-}[ru] &&  & & &{\rm 1} \ar@{-}[lu] \\           
   & & (1) &      } \hskip1truecm
\xymatrix{ & & \save[]+<0mm,0mm>*\frm<6pt>{o}*\txt<3pc>{${\rm 1}$}\restore \save[]+<-12.7mm,0mm>*\txt<2pc>{$\ar@{.}[ldd]$}\restore \save[]+<-12mm,0mm>*\txt<2pc>{$\ar@{.}[ldd]$}\restore \save[]+<-13.5mm,0mm>*\txt<2pc>{$\ar@{.}[ldd]$}\restore& & \\
            {\rm 1}\ar@{-}[rru]\ar@{-}[rrd]& {\rm 1}\ar@{-}[ru]\ar@{-}[rd]& &{\rm 1}\ar@{-}[lu] \ar@{-}[ld]&{\rm 1}\ar@{-}[llu] \ar@{-}[lld]\\
             &   &\save[]+<0mm,0mm>*\frm<6pt>{o}*\txt<3pc>{${\rm 1}$}\restore  & & \\
& & (2) &
}}$$
\noindent
The omnipresent maximal root $v$ displayed on the form (1) has two (marked) exceptional indices.  The unit form $q$ given by diagram (2) accepts a maximal omnipresent root $v$ with two (marked) exceptional indices.

(3)The following quadratic form $q$ accepts the indicated locally maximal omnipresent root $v$ with only one exceptional vertex $j$ and $v(j)=1$. Observe that the form $q^{(j)}$ is not weakly positive:
$${\footnotesize 
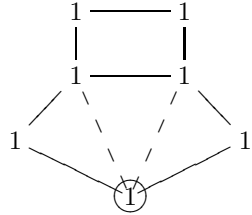
\xymatrix@C12pt@R12pt{&1 \ar@{-}[rr]\ar@{-}[d]&&1\ar@{-}[d]&\\&1 \ar@{-}[rr]\ar@{-}[ld]\ar@{--}[rdd]&&1\ar@{-}[rd]\ar@{--}[ldd] &\\1 \ar@{-}[rrd]&&&& 1 \ar@{-}[lld]\\&&\save[]+<0mm,0mm>*\frm<6pt>{o}*\txt<3pc>{${\rm 1}$}\restore  &&
}}$$
(4) The following example shows that a unit form may accept several maximal omnipresent roots (even with different exceptional vertices):
$$ 
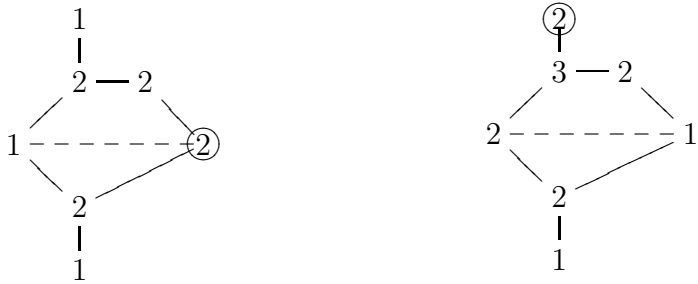
\xymatrix@C12pt@R9pt{ 
& {\rm 1} \ar@{-}[d] & &  \\ 
&{\rm 2} \ar@{-}[r] & {\rm 2} \ar@{-}[rd]& \\
  {\rm 1} \ar@{-}[ru]\ar@{-}[rd]\ar@{--}[rrr] & &  &\save[]+<0mm,0mm>*\frm<6pt>{o}*\txt<3pc>{${\rm 2}$}\restore  \\ 
      & {\rm 2} \ar@{-}[rru]  & & \\   
& {\rm 1} \ar@{-}[u] & &
} \hskip3truecm
\xymatrix@C12pt@R9pt{ 
& \save[]+<0mm,0mm>*\frm<6pt>{o}*\txt<3pc>{${\rm 2}$}\restore \ar@{-}[d] & &  \\ 
&{\rm 3} \ar@{-}[r] & {\rm 2} \ar@{-}[rd]& \\
  {\rm 2} \ar@{-}[ru]\ar@{-}[rd]\ar@{--}[rrr] & &  &{\rm 1}  \\ 
      & {\rm 2} \ar@{-}[rru]  & & \\   
& {\rm 1} \ar@{-}[u] & &
}
$$

(5) We shall provide an example of a root $v$ of the Tits form of a strongly simply connected algebra $A$ such that $v$ is not realizable as $v = \udim X$ for an indecomposable $A$-module $X$.
Consider the algebra $A$ given by the following quiver with relations and the vector $v$ as indicated on the vertices of the quiver:
$${\footnotesize 
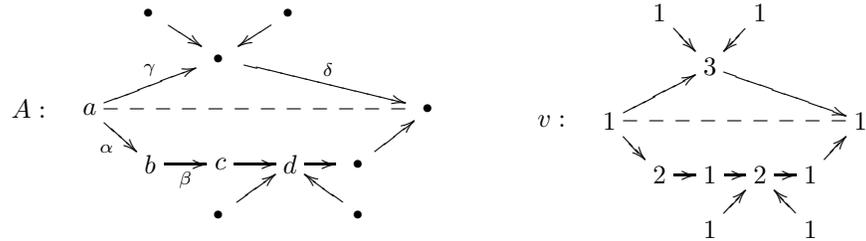
\xymatrix@R8pt@C8pt{&&\bulito \ar[rd]&&\bulito \ar[ld]&&\\
&&&\bulito\ar[rrrd]^{\delta}&&& \\
A:&a\ar[rru]^{\gamma}\ar[rd]_{\alpha}\ar@{--}[rrrrr]&&&&&\bulito\\
&&b\ar[r]_{\beta}&c\ar[r]&d\ar[r]&\bulito \ar[ru]&\\
&&&\bulito\ar[ru]&&\bulito\ar[lu]&} \hskip1cm \xymatrix@R8pt@C8pt{&&1\ar[rd]&&1\ar[ld]&&\\
&&&3\ar[rrrd]&&& \\
v:&1\ar[rru]\ar[rd]\ar@{--}[rrrrr]&&&&&1\\
&&2\ar[r]&1\ar[r]&2\ar[r]&1\ar[ru]&\\
&&&1\ar[ru]&&1\ar[lu]&}}
$$
\noindent
where the dotted line indicates a commutativity relation. We get $q_A(v)=1$, that is, $v$ is a root of the Tits form of $A$. Assume, to get a contradiction,  that there exists an indecomposable $A$-module $X$ such that $\udim X=v$. In case, $X(a) \to X(b) \to X(c)$ is non-zero, there is a decomposition $X(b)=k e_1 \oplus k e_2$ such that $k e_2$ is a direct summand of $X$. Therefore 
$X(\beta)X(\alpha)=0$ and also $X(\delta)X(\gamma)=0$. Hence $X$ restricts to an indecomposable representation $Y$ of the algebra $B$ as in the picture:
$$
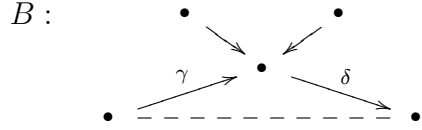
\xymatrix@R8pt@C8pt{B:&&\bulito \ar[rd]&&\bulito \ar[ld]&\\
&&&\bulito\ar[rrd]^{\delta}&& \\
&\bulito \ar[rru]^{\gamma}\ar@{--}[rrrr]&&&&\bulito}$$
which is representation-finite with a preprojective component. Therefore $q_B(\udim Y)=2$ contradicts the main result in \cite{3}.

\subsection{}
Frequently we shall deal with restrictions of quadratic forms and quotients of algebras. The relationship between these operations is clarified in the following proposition.

\begin{propo}
Let $A=kQ/I$ be an algebra and $q = q_A$ the corresponding Tits form. Consider a vertex $a$ of $Q$ and the unit forms: $q'$ the restriction of $q$ to $Q':=Q\setminus \{a\}$ and $\bar q$ the Tits form of the quotient algebra ${\bar A}=A/Ae_aA$. The following statements hold:

(a) ${\bar q} \le q'$ in $\N^{Q_0}$;

(b) if ${\bar q}$ is weakly positive, then $q'$ is weakly positive;

(c) if ${\bar q} \ne q'$ then there is a relation $w + w' \in I(i,j)$, with $i\ne a \ne j$, such that $w$ (resp. $w'$) is a combination of paths passing (resp. not passing) through $a$;

(d) if $v \in \N^{Q_0}$ is a root of $q$ with $v(a)=0$, then $1=q'(v)={\bar q}(v)$;
 
(e) if $v \in \N^{Q_0}$ is an omnipresent root of $q$ such that $q(v,e_a)=1=v(a)$, then $q'={\bar q}$.
\end{propo}
\begin{proof}
Assume that $Q_0=\{1,\ldots, n\}$. First observe that ${\bar A}= kQ'/{\bar I}$, where ${\bar I}=I/A e_aA \cap I$. Let $r_1, \ldots,r_s$ be a minimal set of generators of $I$ in $\cup_{i,j}I(i,j)$ and, for each pair $i,j \in Q$, set $r(i,j)$ the number of those $r_p$ in $I(i,j)$. Write each $r_p=r'_p+r''_p$ as linear combination of paths where $r'_p$ (resp. $r''_p$) is a linear combination of paths not passing (resp. passing) through $a$. Then $r'_1,\ldots,r'_s$ generate ${\bar I}$. We get ${\bar r}(i,j) \le r(i,j)$ for each pair $i,j \in Q_0'$.

This shows (a) and (b) to (d) follow. Only (e) needs an additional argument. The vector $w=v-e_a$ is a root of $q$ with $w(a)=0$. Hence ${\bar q}(w)=1=q'(w)$. The quadratic form ${\hat q}= q' -{\bar q}$ is weakly non-negative by (a) and $w$ is omnipresent in $Q'$ with $0={\hat q}(w)= \sum\limits_{i, j \in Q_0'} w(i)w(j)(r(i,j)-{\bar r}(i,j))$. Therefore $q'={\bar q}$.
\end{proof}

\subsection{} 
We recall that a unit form $q:\Z^n \to \Z$ is {\em critical} (resp. {\em hypercritical}) if $q$ is not weakly positive (resp. not weakly non-negative) but every proper restriction $q|J: \Z^J \to \Z$, for $J$ a proper subset of $\{ 1,\ldots,n \}$, is weakly positive (resp. weakly non-negative). Clearly, a unit form is weakly positive (resp. weakly non-negative) if and only if it does not accept a critical (resp. hypercritical) restriction.

Let $q(x_1,\ldots,x_n)$ be a critical unit form. If $n \ge 3$, a theorem of Ovsienko, see \cite{21}, says that $q$ is non-negative and there is an omnipresent vector $z$ with positive integral coordinates and $q(z)=0$; a minimal such vector is called a {\em critical vector}. If $n=2$, then $q_{12} \le -2$ (only if $q_{12}=-2$ we get a critical vector $(1,1)$). The critical forms and critical vectors have been classified \cite{21}, it is important to observe that for a critical vector $z$ there is always an index $i$ with $z(i)=1$.

In \cite{18} it was shown that given a hypercritical unit form $q(x_1,\ldots,x_n)$, there is an index $1 \le j \le n$ such that the restriction $q|J$ is critical for $J= \{ 1,\ldots,n \} \setminus \{ j\}$; moreover, for any index $1 \le i \le n$ such that the restriction $q|I$ is critical for $I= \{ 1,\ldots,n \} \setminus \{ i\}$, then the corresponding critical vector $z_i$, as an element of $\Z^n$, satisfies $q(z_i,e_i) < 0$. We write $q^{(i)}=q|I$ for $I= \{ 1,\ldots,n \} \setminus \{ i\}$. The following statement is a reformulation of \cite{18}, Proposition (1.4), we present a sketch of its proof.

\begin{lemma}
Let $q(x_1,\ldots,x_n)$ be a slender hypercritical form. Then one of the following two situations occur:

(1) there are positive vectors $v$ and $w$ satisfying $q(v)=-1$ and $q(w)=-3$;

(2) $n=4$ and $q=q_M$ is the form associated to the quiver:
$$
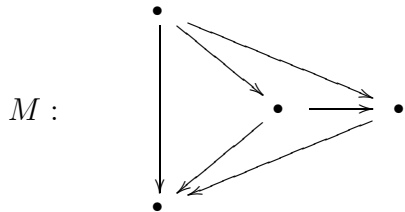
\xymatrix{  &\bulito \ar[dd]\ar[rd]\ar[rrd]& & \\
             M: & & \bulito \ar[r]\ar[ld] & \bulito \ar[lld]\\
              & \bulito & &
}$$
\noindent In this case, there are positive vectors $v$ and $w$ satisfying $q(v)=-2$ and $q(w)=-3$. 
\end{lemma}
\begin{proof}
Observe that $n \ge 3$ and consider the critical restriction $q'=q^{(n)}$ with critical vector $z$ considered as an element of $\Z^n$ and $q(z,e_n)=-m$, a negative integer. Assume $z(1)=1$.

(a) Show $0 < m \le 3$. Indeed, the vector $v=z-e_1+e_n$ is not omnipresent, therefore
$$0 \le q(v)=q(z)+2-q(z,e_1)+q(z,e_n)-q(e_1,e_n)= 2 - m -q(e_1,e_n) \le 3 -m.$$

(b) For $m=1$, we get $q(2z+e_n)=-1$ and $q(4z+e_n)=-3$.

(c) For $m=2$, we get $q(z+e_n)=-1$ and $q(2z+e_n)=-3$.

(d) For $m=3$, we get $q(v)=0$ and $q(e_1,e_n)=-1$. Hence $q^{(1)}$ is critical and $v$ is a critical vector (since $v(n)=1$). We may assume $q(v, e_1)=-r$ is a negative integer. If $r=1$ or $2$ we get by (b) and (c) vectors $v'$ and $w'$ satisfying $q(v')=-1$ and $q(w')=-3$ and we are done. Therefore, we may assume $r=3$ and as above get $q(e_2,e_1)=-1$. Repeating the argument, we may assume $q(e_2,e_n)=-1$ which implies that $q|\{1,2,n \}$ is critical of type ${\tilde A}_2$. Hence $n=4$ and by symmetry $q=q_M$. Consider $v= (1,1,1,1)$ and $w=(2,2,1,1)$ to get $q(v)=-2$ and $q(w)=-3$.
\end{proof}

\subsection{}
We specialize the main result of \cite{18} into the following statement which is fundamental in our considerations.

\begin{theor}
Let $q: \Z^n \to \Z$ be a slender unit form accepting a maximal omnipresent root. Then $q$ is weakly non-negative.
\end{theor}
\begin{proof} 
Let $v$ be a maximal omnipresent root of $q$. Assume $q$ is not weakly non-negative, then there is a restriction $q'=q|J: \Z^J \to \Z$ which is {\em hypercritical}. First suppose that there are non-negative vectors $w_1,w_2 \in \Z^J$ such that $q'(w_1)=-1$ and $q'(w_2)=-3$. We have $q(v,w_1) = 0$. Indeed, otherwise by (1.2), we had $q(v,w_1)>0$ then the vector $w:= v+q(v,w_1)w_1$ which is strictly bigger than $v$ satisfies $q(w) = 1+ q(v,w_1)^2 +q(v,w_1)^2q'(w_1) =1$, contradicting the maximality of $v$. Since the support of $w_1$ and $w_2$ is the same $J$, then also $q(v,w_2)=0$ and the vector $w':=2v+w_2$ yields the final desired contradiction.
                                 
Next suppose that $q'=q_M$ and consider $w_1,w_2 \in \Z^{ \{1,2,3,4 \} }$ such that $q'(w_1)=-2$ and $q'(w_2)=-3$. Observe that $0 \le q(v,w_1) \le 2$. In case $q(v,w_1)=0$, then also $q(v,w_2)=0$ and $q(2v+w_2)=1$, a contradiction. In case $q(v,w_1)=1$, then $q(v+2w_1)=-1$ and we may switch to the first case. In case $q(v,w_1)=2$, then $q(v+2w_1)=1$ which is a contradiction. 
\end{proof}

\section{On the realization of roots}
\subsection{}

A {\em quiver\/} $Q$ is an oriented graph with set of vertices $Q_0$
and set of arrows $Q_1$. The {\em path algebra\/} $kQ$ has as
$k$-basis the oriented paths in $Q$, including a trivial path $e_s$
for each vertex $s\in Q_0$, with the product given by concatenation
of the paths. A module $X\in \mod_{kQ}$ is a {\em representation\/}
of $Q$ with a vector space $X(s)=e_sX$ for each vertex $s\in Q_0$ and
a linear map $X(\alpha)\colon X(s)\to X(t)$ for each arrow $s\
\aderecha{\alpha}\ t$ in $Q_1$.

For a finite-dimensional $k$-algebra $A=kQ/I$, the quiver
$Q$ is defined in the following way: the set of vertices $Q_0$ is the set of
isoclasses of simple $A$-modules $\{1,\ldots,n\}$. Let $S_i$ be a
simple $A$-module representing the $i$-th trivial path $e_i$. Then there are as
many arrows from $i$ to $j$ in $Q$ as $\dim_k\Ext^1_A(S_i,S_j)$. If $A$ is triangular, by a result of Bongartz \cite{3}, we can select $r_{ij}=\dim_k\Ext^2_A(S_i,S_j)$ linear combinations 
$f_{ij}^{(1)}, \ldots, f_{ij}^{(r_{ij})}$ of paths from $i$ to $j$ in such a way that 
$\{ f_{ij}^{(d)}: i,j \in Q_0 {\rm \, \, and \,\,}1 \le d \le r_{ij} \}$ forms a minimal set of generators of $I$.

We shall identify $A=kQ/I$ with a $k$-category whose objects are the
vertices of $Q$ and whose morphism space $A(s,t)$ is $e_tAe_s$. We
say that $B$ is a {\em convex subcategory\/} of $A$ if $B=kQ'/I'$ for
a path closed subquiver $Q'$ of $Q$ and $I'=I\cap kQ'$. In this view,
an $A$-module $X$ is a $k$-linear functor $X\colon A\to \mod_k$. The
{\em dimension vector\/} of $X$ is $\udim\,X=(\dim_kX(s))_{s\in
Q_0}\in \N^{Q_0}$ and the {\em support\/} of $X$ is the set of vertices 
$\supp\,X=\{s \colon X(s)\ne 0\}$.

For an algebra $A$, we consider the standard duality $D\colon
\mod_A\to \mod_{A^{op}}$ defined as $D=\Hom_k(-,k)$, where $A^{op}$
is the opposite algebra of $A$. 

\subsection{} 
Let $v\in \N^{Q_0}$ be a dimension vector. The {\it variety of
$A$-modules} of dimension vector $v$ is the closed subset $\mod_A (v)$ of
the affine space $\prod\limits_{(i \aderecha{\alpha}\ j)\in Q_1} k^{v(i)v(j)}$\ , 
formed by all tuples
$M=(M(\alpha ))_{\alpha \in Q_1}$ such that for any element $\rho =
\sum\limits^r_{s=1}\lambda_s\alpha^{(s)}_{j_s}\cdots
\alpha^{(s)}_{j_1}\in I(i,j)$ the $v(i)\times v(j)$-matrix $M(\rho ) =
\sum\limits^r_{s=1}\lambda_s M(\alpha^{(s)}_{j_s})\cdots
M(\alpha^{(s)}_{j_1})$ is zero (see \cite{28}). The affine algebraic group
$G(v)=\prod\limits_{i\in Q_0}GL_k(v(i))$
acts on the variety $mod_A (v)$ in such a way that two points
belong to the same orbit if and only if they are isomorphic. Clearly,
$\dim \ G(v)=\sum\limits_{i\in Q_0} v(i)^2$.

Let $v\in \N^{Q_0}$ be a dimension vector. Let $C$ be an
irreducible component of $mod_A (v)$. A decomposition $v=w_1 +
\cdots +w_s$ with $w_s\in \N^{Q_0}$ is called a {\it generic
decomposition in} $C$ if the set
$$\{X\in C\colon X=X_1 \oplus \cdots \oplus X_s\ \hbox{with}\
\udim \ X_i=w_i\ \hbox{and}\ X_i\ \hbox{indecomposable}\}$$
\noindent 
contains an open and dense subset of $C$. Given an irreducible component $C$ of $\mod_A (v)$,
there always exists a unique generic decomposition in $C$, say $v=w_1 +
\cdots +w_s$, and there are irreducible components $C_i$ of $\mod_A (w_i)$
such that the generic decomposition in $C_i$ is irreducible and the
following inequality holds:
$$\dim \ G(z)-\dim \ C\ge \sum^s_{i=1}(\dim \ G(w_i)-\dim \ C_i),$$
\noindent 
\cite{23}, \cite{25}. Moreover, generically there are no extensions between modules in $C_i$ and $C_j$ for $i \ne j$, see \cite{9}, \cite{23}. 

\subsection{}
Let $A$ be a triangular algebra and consider the bilinear form
$$\langle v,w \rangle_A = \sum_{i\in Q_0}v(i)w(i) -\sum_{(i\to j)\in Q_1}v(i)w(j)
+\sum_{i,j\in Q_0}r(i,j)v(i)w(j)$$
whose associated quadratic form is $q_A$. Recently, the authors showed the following result \cite{30'}.

\begin{theor}
For any two $A$-modules $X, Y$ over a triangular algebra $A$ the inequality
$$\langle \udim \, X, \udim \, Y \rangle_A  \ge \dim_k \Hom_A(X,Y) - \dim_k \Ext^1_A(X,Y)$$
holds.
\end{theor}

\subsection{}
The following, probably well-known, result is relevant in our context.

\begin{propo}
Let $A=kQ/I$ be a triangular algebra and $v$ be a root of the Tits form $q_A$ satisfying the following conditions:

(a) $q_A$ is weakly non-negative;

(b) there is a non-negative vector $u \in \Z^n$ such that for any non-negative vector $z \le v$ with $q_A(z)=0$, then $z$ is an integral multiple of $u$.

Then there exists a number $r \in \N$ such that $0 \le v-ru$ and there is an indecomposable realization $X \in \mod_A(v-ru)$. In case, $v$ is a maximal positive root, then there is an indecomposable realization $X \in \mod_A(v)$.
\end{propo}
\begin{proof}   
Let $A=kQ/I$ be a triangular algebra and $v$ a root of the Tits form $q_A$ satisfying (a) and (b). Take $C$ an irreducible component of $\mod_A (v)$ of maximal dimension and $v=w_1 + \ldots +w_s$ the generic decomposition. Let $C_i$ be an irreducible component of $\mod_A (w_i)$ such that the generic decomposition in $C_i$ is irreducible. Moreover we may define open subsets $U_i$ of $C_i$ where each $X_i \in U_i$ is extension-orthogonal to any $X_j \in U_j$, for $j \ne i$. For any choice $X_i$ of indecomposable modules in $U_i$, $i= 1,\ldots, s$, we get
$$1=q_A(v)= \sum\limits_{i=1}^s q_A(w_i)+ \sum\limits_{i \ne j}\langle w_i,w_j \rangle \ge \sum\limits_{i=1}^s q_A(w_i)+ \sum\limits_{i \ne j}\dim_k \Hom_A(X_i,X_j),$$
\noindent  
where all the summands are non-negative. We distinguish two cases.

(1) There is one $w_i$ such that $q_A(w_i)=1$, say $i=1$. Suppose $s> 1$, then $q_A(v,w_2)=0$ and
$q_A(w_2)=0$. We may suppose that $w_2=r u$ for certain $r \in \N$. Then $w_1=v-r u$ accepts a realization $X_1 \in \mod_A(w_1)$. In case $v$ is maximal and $s>1$ then $q_A(v+w_2)=1$ yields a contradiction. Therefore $v$ maximal root implies $s=1$.

(2) All of the $w_i$ satisfy $q_A(w_i)=0$. For some $i \ne j$ we have $q_A(w_i,w_j) = \dim_k\Hom_A(X_i,X_j) =1$, for modules $X_i \in U_i$ and $X_j \in U_j$, say $i=1,j=2$ and for all others $i \ne 1,2$ and any $j$ the modules in $U_i$ and $U_j$ are Hom-orthogonal.  
Then $s > 1$ and there are constants $r_1, r_2$ such that $w_1=r_1u$ and $w_2=r_2u$. Then $1 = q_A(w_1,w_2)=r_1r_2q_A(u,u)=0$ which is a contradiction. Therefore situation (b) never occurs.
\end{proof}

\begin{coro}
Let $A$ be a triangular algebra whose Tits form is weakly non-negative and $v$ be a root of $q_A$ such that for every $0 \le u \le v$ we have $q_A(u) \ge 1$. Then there exists an indecomposable realization $X \in \mod_A(v)$.
\end{coro}
\begin{proof}   
Condition (b) above is satisfied for $u=0$.
\end{proof}

\begin{coro}
Let $A$ be a triangular algebra whose Tits form is weakly positive and $v$ be any root of $q_A$. Then there exists an indecomposable realization $X \in \mod_A(v)$.
\end{coro}

\subsection{}
We recall important classes of algebras where roots may be realized:

(a) Representation-finite triangular algebras. Indeed such an algebra $A$ has a weakly positive Tits form.

(b) Let $A=kQ/I$ be a basic finite dimensional $k$-algebra. A module
$_AT$ is called a {\em tilting module\/} if it satisfies:

(T1) $\Ext^2_A(T,-)=0$; (T2) $\Ext^1_A(T,T)=0$; (T3) The number of non-isomorphic indecomposable direct summands of $_AT$ is the rank of the Grothendieck group $K_0(A)$.

In case $A=k\Delta$ is a hereditary algebra and $_AT$ is a tilting
module, $B=\End_A(T)$ is called a {\em tilted algebra\/} of type
$\Delta$. The work of Kerner \cite{20} shows that a tilted algebra $B$ is tame if an only if the Tits form $q_B$ is weakly non-negative. In that case all roots of $q_B$ can be realized as dimension vectors of indecomposable modules.

\begin{propo}
Let $A$ be a triangular algebra accepting a locally maximal omnipresent root $v$ of $q_A$ with exceptional indices $a \ne b$. Then there is an indecomposable $A$-module $Y$ such that $y=\udim Y$ is a positive root of $q_A$ satisfying:

(i) $y \le v$ and $y(a)=1$;

(ii) $q_A(v-y)=0$;

(iii) if $y(b) \ne 0$ then $v=y$.
\end{propo}
\begin{proof}
By (1.3), there is a positive root $y$ of $q_A$ such that there exist a sequence of reflections $\sigma_{i_1},\ldots,\sigma_{i_s}$ satisfying:

(1) $\sigma_{i_s}\ldots \sigma_{i_1}(e_a)=y \le v$;

(2) $\sigma_{i_t}\ldots
\sigma_{i_1}(e_j)= e_j +e_{i_1}+\cdots +e_{i_t}$ for all $1 \le t \le s$;

(3) for every $0 \le u \le y$ we have $q_A(u) \ge 1$;

(4) $q_A(v-y)=0$.

By (2.3), there is a realization of $y$ as desired. In case $y(b) \ge 1$ then $v-y$ is a non-negative vector with $(v-y)(a)=0=(v-y)(b)$ and $q_A(v-y)=0$. Since $q_A^{(a,b)}$ is weakly positive, then $v-y=0$.
\end{proof}

\subsection{}
Let $B$ be an algebra, $\mathcal C$ be a standard component of $\Gamma_B$
and $X$ be an indecomposable module in $\mathcal C$. In \cite{1}, three
{\it admissible operations\/} (ad~1), (ad~2) and (ad~3) were defined
depending on the shape of the support of $\Hom_B(X,-)\vert_{\mathcal C}$
in order to obtain a new algebra $B'$.

\begin{itemize}
\item[(ad~1)] If the support of $\Hom_B(X,-)\vert_{\mathcal C}$ is of the
form 
$$X=X_0\to X_1\to X_2\to \cdots$$
we set $B'=(B\times D)[X\oplus Y_1]$, where $D$ is the full $t\times
t$ lower triangular matrix algebra and $Y_1$ is the indecomposable
projective-injective $D$-module.

\item[(ad~2)] If the support of $\Hom_B(X,-)\vert_{\mathcal C}$ is of the
form 
$$Y_t\leftarrow \cdots \leftarrow Y_1\leftarrow
X=X_0\to X_1\to X_2\to \cdots$$
with $t\ge 1$, so that $X$ is injective, we set $B'=B[X]$.

\item[(ad~3)] If the support of $\Hom_B(X,-)\vert_{\mathcal C}$ is of the
form 
$${\footnotesize
\xymatrix@C6pt@R5pt{
   &Y_1      &\to &Y_2      &\to &\cdots &\to &Y_t      &    &    &
   &\cr
   &\uparrow &    &\uparrow &    &       &    &\uparrow &    &    &
   &\cr
X= &X_0      &\to &X_1      &\to &\cdots &\to &X_{t-1}  &\to &X_t &\to &\cdots
\cr}}$$
with $t\ge 2$, so that $X_{t-1}$ is injective, we set $B'=B[X]$.
\end{itemize}

In each case, the module $X$ and the integer $t$ are called,
respectively, the {\it pivot\/} and the {\it parameter of the
admissible operation}. 

The dual operations are denoted by (ad~1*), (ad~2*) and (ad~3*).

Following \cite{1}, an algebra $A$ is a {\it coil enlargement\/} of
the critical algebra $C$ if there is a sequence of algebras
$C=A_0,A_1,\ldots,A_m=A$ such that for $0\le i<m$, $A_{i+1}$ is
obtained from $A_i$ by an admissible operation with pivot in a stable
tube of $\Gamma_C$ or in a component (coil) of $\Gamma_{A_i}$
obtained from a stable tube of $\Gamma_C$ by means of the admissible
operations done so far. When $A$ is tame, we call $A$ a {\it coil
algebra}. A typical example of a coil algebra is the following
bound quiver algebra given by the quiver 

$${\footnotesize 
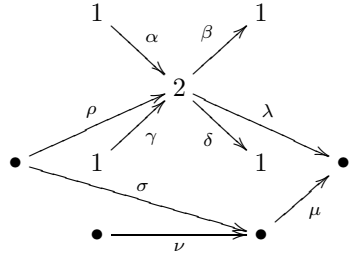
\xymatrix@C20pt@R16pt{&1 \ar[dr]^{\alpha}&&1 &\\&&2 \ar[ru]^{\beta}\ar[rd]_{\delta}\ar[drr]^{\lambda}&&\\\bullet \ar[urr]^{\rho}\ar[drrr]^{\sigma}&1 \ar[ur]_{\gamma}&&1 &\bullet\\&\bullet \ar[rr]_{\nu}&&\bullet \ar[ru]_{\mu}& 
}}$$
\noindent
with relations $\lambda \alpha=0, \, \lambda \gamma=0, \, \beta \rho=0, \,\delta \rho=0$ and $\lambda \rho= \mu \sigma$. Observe that the dimension vector $u$ displayed with support in the critical algebra $C$ satisfies $q_A(u,-)=0$.

\begin{propo}
Let $A$ be a representation-infinite coil algebra, then $A$ does not accept maximal positive roots of $q_A$.
\end{propo}
\begin{proof}
Suppose $v$ is a maximal positive root of $q_A$.
Let $C$ be a convex critical subcategory of $A$ with $z$ a critical vector. Hence $q_C(z)=0$. Let $i$ be a vertex of $Q$, we shall prove that $q_A(z,e_i)=0$. Indeed, if $z(i)>0$ then $q_A(z,e_i)=0$, since otherwise $q_A(z-e_i) <0$. Suppose $z(i)=0$ and there is no arrow between $i$ and vertices in supp $z$ in $Q$, then by the definition of the (ad)-operations defining coils \cite{1}, there are no relations between $i$ and vertices in supp $z$. Hence $q_A(z,e_i)=0$. Finally, suppose $z(i)=0 < z(j)$ and there is an arrow between $i$ and $j$, say $i \to j$. Consider $R$ the restriction of the indecomposable projective $A$-module $P_i$ to $C$ and $R(j) \ne 0$. Since $A$ is a coil algebra, then $R$ is an indecomposable regular $C$-module. Hence $q_A(z,e_i)=q_A(z,\udim \,P_i) -q_A(z,\udim \, R)=z(i) + \langle z, \udim \, P_i\rangle =0$. Therefore $q_A(z,-)=0$, in particular $q_A(z,v)=0$, which implies that $v+z$ is a root, contradicting the maximality of $v$. 
\end{proof}

\subsection{Proof of Theorem 1} 
Let $A=kQ/I$ be a triangular algebra without double arrows and whose Tits form $q_A$ accepts a maximal omnipresent root $v$. Hence $q_A$ is weakly non-negative by (1.8). We shall show that there exists an indecomposable module $X$ such that $v= \udim X$. For this purpose, let $C$ be an irreducible component of $\mod_A (v)$ of maximal dimension and $v=w_1 + \ldots +w_s$ the generic decomposition. Let $C_i$ be irreducible components of $\mod_A (w_i)$ such that the generic decomposition in $C_i$ is irreducible. Moreover we may define open subsets $U_i$ of $C_i$ where each $X_i \in U_i$ is extension-orthogonal to any $X_j \in U_j$, for $j \ne i$. The following inequalities hold, as consequence of the hypothesis and (2.4), for any choice $X_i$ of indecomposable modules in $U_i$, $i= 1,\ldots, s$:
$$1=q_A(v)= \sum\limits_{i=1}^s q_A(w_i)+ \sum\limits_{i \ne j}\langle w_i,w_j \rangle \ge \sum\limits_{i=1}^s q_A(w_i)+ \sum\limits_{i \ne j}\dim_k \Hom_A(X_i,X_j),$$
\noindent  
where all the summands are non-negative. We distinguish two cases.

(a) There is one $w_i$ such that $q_A(w_i)=1$, say $i=1$. Suppose $s> 1$, then $q_A(v,w_2)=0$ and the construction of the root $v+w_2$ yields a contradiction to the maximality of $v$. Therefore $v=w_1$ and $X=X_1$ is the desired realization of $v$.

(b) All of the $w_i$ satisfy $q_A(w_i)=0$. For some $i \ne j$ we have $q_A(w_i,w_j) = \dim_k\Hom_A(X_i,X_j) =1$, for modules $X_i \in U_i$ and $X_j \in U_j$, say $i=1,j=2$ and for all others $i \ne 1,2$ and any $j$ the modules in $U_i$ and $U_j$ are Hom-orthogonal.  
If $s > 2$, we get a contradiction since then $v+w_3$ is a root. Assume that $s=2$. 
In case there $v$ accepts only one exceptional index $a$ then either $q_A(w_1,e_a)=0$ or $q_A(w_2,e_a)=0$ which, again, contradicts the maximality of $v$. Therefore,
by (1.4), there exist two exceptional indices $a,b$ for $v$ such that $q_A(e_a,w_1)=1=q_A(e_b,w_2)$ and $v(a)=1=v(b)$. By (2.5), there is an indecomposable $A$-module $Y$ such that $y=\udim Y$ is a positive root of $q_A$ satisfying:

(i) $y \le v$ and $y(a)=1$;

(ii) $q_A(v-y)=0$;

(iii) if $y(b) \ne 0$ then $v=y$.

Consider the non-negative isotropic vector $w=v-y$. In case $w(b)=0$ then (iii) applies and yields a realization $Y$ of $v$. Hence we may assume that $w(b)=1$. Therefore $1=q_A(v)=q_A(w+y)=1+q_A(w,y)$ and $q_A(w,v)=q_A(w,w+y)=0$. This implies that $q_A(v+w)=1$, again a contradiction to the maximality of $v$. This completes the proof of Theorem 1.

\section{More examples}

\subsection{} 
Consider the algebras given by the following quiver with relations:
$$\footnotesize{ B_{\lambda \mu}: 
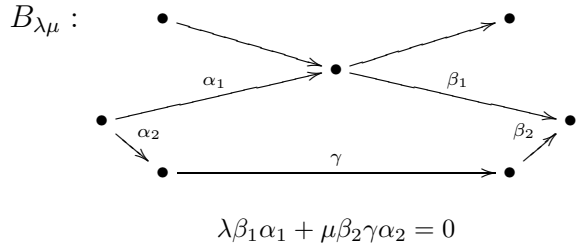
\xymatrix@C12pt@R9pt{ & \bullet \ar[rd] & &\bullet &\\
             &   &\bullet \ar[ru]\ar[rrd]^{\beta_1} &  & \\
           \bullet \ar[rru]^{\alpha_1}\ar[rd]^{\alpha_2} & & & & \bullet \\ 
             &\bullet  \ar[rr]^{\gamma}& & \bullet \ar[ru]^{\beta_2}&\\
                 &&   \lambda \beta_1 \alpha_1 +\mu \beta_2 \gamma \alpha_2 =0           &&
}}
$$
and consider $B_{11}, B_{10}$ and $B_{01}$. The three algebras have the same weakly positive form 
$q$ which has a maximal omnipresent root $v$. While there is a one-to-one correspondence between the isoclasses of indecomposable $B_{11}$-modules and the roots of $q$ (in fact, $B_{11}$ is a representation-finite tilted algebra), there are infinitely many indecomposable $B_{01}$-modules $Y$ with the same dimension vector $y$:

 $$\footnotesize{\xymatrix@R0pt@C2pt{&&1&&1& \\& &&1&& \\v:&1&&&&1 \\ &&1&&1&}}
\hskip2truecm
\footnotesize{\xymatrix@R0pt@C2pt{&&1&&1& \\ &&&2&& \\y:&1&&&&1 \\ &&0&&0&}}$$ 
\smallskip

Observe that $B_{01}$ is wild while the algebra $B_{10}$ is tame not of polynomial growth.

\subsection{} 
By results in \cite{26} the {\em number of parameters} $p(A)$, defined as the number of convex critical subcategories of $A$, of a tame algebra $A$ having a sincere indecomposable module is at most $2$. Clearly, if $p(A)=0$, then $A$ is representation-finite and if $p(A)=1$ then $A$ is a finite enlargement or coenlargement of a representation-infinite domestic tilted convex subcategory of $A$ (for definitions see \cite{Ri0}). In case $p(A)=2$ then $A$ is the glueing of two representation-infinite domestic tilted algebras. We recall that Bongartz classified the families of representation-finite algebras $A$ with the above properties and $n \ge 13$. Dr\"axler \cite{10} built the small cases ($n \le 12$) with the help of a computer program. The algebras with $p(A)=2$ and $n \ge 20$ were classified by de la Pe\~na in \cite{27}. 

Consider the algebras given by quivers with relations:
$$\footnotesize{A_1:\xymatrix@C5pt@R5pt{
&&&&\bulito \ar[rd]&&&&&\\
&\bulito\ar[r]&\bulito\ar[r]&\bulito\ar[ru]\ar[rd]\ar@{--}[rr]&&\bulito \ar[r]&\ldots&\bulito\ar[rdd]&&\\
&&&&\bulito \ar[ru]&&&&&\\
\bulito\ar[ruu]\ar[rrrrd]\ar@{--}[rrrrrrrr]&&&&&&&&o&\\
&&&&\bulito\ar[rrrru]&&&&&&\\
&&&&\bulito\ar[u]&&&&&\\
&&&&\vdots&&&&&\\
&&&&\bulito\ar[u]&o\ar[l]\ar@{--}[rrruuuu]&&&&\\
&&&&\bulito \ar[u]&&&&&
}
\hskip.5truecm
\xymatrix@C5pt@R5pt{
&&&&1&&&&&\\
&1&1&1&&1&\ldots&1&&\\
1&&&&1&&&&1&\\
&&&&2&&&&&&\\
&&&&2&&&&&\\
&&&&\vdots&&&&&\\
&&&&2&1&&&&\\
&&&&1&&&&&
}
}$$
$$ \footnotesize{A_2: \xymatrix@C5pt@R5pt{ 
     &&\bullet \ar[r] &\bullet \ar[r] &\bullet \ar[r]& \bullet\ar[rd]\ar[r]& \bullet \\
  \bullet \ar[r]&\bullet \ar[ru]\ar[rd]\ar@{--}[rrrrr] & & & &&\bullet  \\ 
     & & \bullet \ar[rr]&   & {\rm o} \ar[rru]  & &   
}
\hskip.5cm
 \xymatrix@C8pt@R6pt{ 
     &&2 &2 &2& 2& 1 \\
 v: 1&2 & & & &&1  \\ 
     & & 2&   & 2  & &   
}}
$$
The algebra $A_1$ represents in fact a family of algebras parametrized by the number of vertices $n$. The algebras in the family are tilted of type ${\tilde D}_n$ with a preprojective component containing a sincere module with the indicated maximal omnipresent root $v$ of $q_{A_1}$, that is, $p(A_1)=1$. Observe that $v$ has two exceptional indeces marked by ${\rm o}$. For the algebra $A_2$, the Tits form $q_{A_2}$ accepts a maximal root $v$ with a unique exceptional index marked by ${\rm o}$. This algebra is tame tilted and contains two tame concealed subcategories of types ${\tilde E}_7$ and ${\tilde E}_8$, that is, $p(A_2)=2$. 

\section{Tilted algebras}

\subsection{} We start recalling from \cite{29a}, \cite{30} and \cite{7} the following useful version of the {\em Splitting Lemma}.
We say that an indecomposable $A$-module $X$ is {\em extremal} if ${\rm supp }X=\{ i \in Q_0 : X(i) \ne 0 \}$ contains all sinks and sources of $Q$. 

\begin{lemma}
\label{splittinglemma}%
Let $A$ be a triangular algebra and $B=B_0,$ $B_1, \ldots, B_s=A$
a family of convex subcategories of $A$ such that, for each $0\le
i\le s$, $B_{i+1}=B_i[M_i]$ or $B_{i+1}=[M_i]B_i$ for some
indecomposable $B_i$-module $M_i$. Assume that the category $\ind_B$ 
of indecomposable $B$-modules admits a splitting
$\ind_B={\mathcal P} \lor {\mathcal J}$, where ${\mathcal P}$ and
${\mathcal J}$ are full subcategories of $\ind_B$ satisfying the
following conditions:

\begin{itemize}
\item[ {\rm (S1)}]
$\Hom_B({\mathcal J},{\mathcal P})=0$;

\item[ {\rm (S2)}]
 for each $i$ such that $B_{i+1}=B_i[M_i]$, the
restriction $M_{i}\big|_B$ belongs to the additive category
 add $({\mathcal J})$;

\item[ {\rm (S3)}]
 for each $i$ such that $B_{i+1}=[M_i]B_i$,
the restriction $M_{i}\big|_B$ belongs to the additive category
 add $({\mathcal P})$;

\item[ {\rm (S4)}]
 there is an index $j$ with $B_{j+1}=B_j[M_j]$ and $M_j\in {\mathcal J}$.
\end{itemize}

In case $A$ accepts an indecomposable extremal module, then $A$ is a tilted algebra.
\end{lemma}

\subsection{Proof of Theorem 2}

Let $A$ be a strongly simply connected algebra and assume that $v$ is a maximal omnipresent root of $q_A$. Hence $q_A$ is weakly non-negative by Theorem 0 and \cite{7} implies that $A$ is a tame algebra. Moreover, part (a) of Theorem 1 implies the existence of an indecomposable module $X$ with $\udim \, X= v$. In particular, $X$ is an {\em extremal} module. We distinguish two cases:

(a) $A$ is representation-finite. Since $A$ is strongly simply connected then $\Gamma_A$ is a preprojective component. Since $A$ is sincere, then $A$ is tilted.

(b) $A$ is representation-infinite. Let $C$ be a convex critical subcategory of $A$ and $z$ be a positive generator of rad $q_C$ as a vector in $\Z^{Q_0}$. Consider $B$ be a maximal convex coil extension of $C$. Let $ind_B= {\mathcal P}\lor {\mathcal J}$ be a splitting of $ind_B$ such that ${\mathcal J}$ is the preinjective component of the Auslander-Reiten quiver of $B$. Consider any sequence $B=B_0,$ $B_1, \ldots, B_s=A$ of convex subcategories of $A$ such that, for each $0\le i\le s$, $B_{i+1}=B_i[M_i]$ or $B_{i+1}=[M_i]B_i$ for some indecomposable $B_i$-module $M_i$. 

We shall prove that we get a splitting situation as in the above Lemma. Indeed, (S1) is satisfied. For (S2), consider $B_{i+1}=B_i[M_i]$. Consider a decomposition $M_i|_B = N_+ \oplus N_0 \oplus N_-$ such that $N_+$ lies in the additive closure of the preprojective component of $B$ (resp. $N_0$ in the additive closure of the coils of $B$; $N_-$ in the additive closure of ${\mathcal J}$). In case $N_+ \ne 0$, then $A$ has as quotient $B[N_+]$ which is of wild representation type, \cite{29a}. In case $N_0 \ne 0$ then $B[N_0]$ is a coil algebra, contradicting the maximality of $B$. Hence $M_i|_B$ belongs to the additive closure of ${\mathcal J}$. For (S3), consider $B_{i+1}=[M_i]B_i$ and proceed as in (S2) to show that $M_i|_B$ belongs to the additive closure of ${\mathcal P}$. Finally if (S4), or its dual, is not satisfied then $A=B$ is a coil extension of $C$. Proposition (2.6) shows that $A$ does not accept a maximal omnipresent root which contradicts the existence of $v$.

Since $X$ is an indecomposable extremal module, then the Spliting Lemma implies that $A$ is tilted.

\subsection{} 
Let $A$ be a strongly simply connected algebra whose Tits form $q_A:\Z^n \to \Z$ accepts a maximal omnipresent root $v$. Then $A$ is a tame tilted algebra with a directing component ${\mathcal C}$ and an indecomposable sincere module $X$ such that $\udim X=v$. 

First we observe that $X \in {\mathcal C}$, in particular pdim$_A X \le 1$ and idim$_A X \le 1$. Indeed, if $X \notin {\mathcal C}$, we may assume that there are morphism $\Hom_A(Y,X) \ne 0$ for some $Y \in {\mathcal C}$ and therefore ${\mathcal C}$ does not contain injective modules. Since $A$ is tame then $A$ is tilted of Euclidean type and ${\mathcal C}$ is a preprojective component, see \cite{31}. Therefore $A$ does not accept a maximal omnipresent root, a contradiction.

Consider a {\em slice} ${\mathcal S}$ of ${\mathcal C}$ containing $X$. By a result of Happel \cite{17}, the Hochschild cohomology $H^1({\mathcal S})=H^1(A)=0$, since $A$ is strongly simply connected. Therefore ${\mathcal S}$ is a tree and we may assume that $X$ is a unique source in ${\mathcal S}$ (since all indecomposable injective modules are successors of $X$ in $\Gamma_A$. Consider the number $t({\mathcal S})$ of terminal vertices of the tree ${\mathcal S}$ and $s(X)$ the number of neighbours of $X$ in ${\mathcal S}$. 

If there are no projective modules in ${\mathcal C}$, then the slice ${\mathcal S}$ is of Euclidean type \cite{25} and therefore there are infinitely many omnipresent roots for $q_A$, a contradiction. Therefore there is a last projective $P_s$ in the order of paths in ${\mathcal C}$, we shall denote $R_0=$ rad $P_s$, where $s$ is a source in the quiver of $A$, and we may assume that the algebra $B=A/(s)$ is connected and the $B$-module $R_0$ is indecomposable (if not so, we consider the dual algebra $A^{op}$). In particular $A=B[R_0]$, where $B$ is a tilted algebra with a directing component ${\mathcal C}'$ accepting a slice ${\mathcal S}'={\mathcal S} \setminus \{ P_s^{\tau} \}$ with $t({\mathcal S}')= t({\mathcal S})-1$ if $s(R_0) >1$ or $t({\mathcal S}')=t({\mathcal S})$ if $s(R_0)=1$. Since $X$ is sincere and ${\mathcal C}'$ is directing, no module of the slice ${\mathcal S}'$ is injective with the possible exception of $R_0$ itself (in that case, $P_s$ is projective-injective and therefore the unique sincere $A$-module). Moreover, since $A$ is tame, the algebra $B$ is tame and the vector space category ${\mathcal U}(\Hom_B(R_0, \mod_B))$ is tame. In particular, we get (see \cite{Ri0}):

(a) for every indecomposable $B$-module $Y$ we have $\dim_k \Hom_B(R_0,Y) \le 2$;

(b) if for some indecomposable $B$-module $Y$ we have $\dim_k \Hom_B(R_0,Y)= 2$, then for any indecomposable $B$-module $Z$ we have either $\Hom_B(Y,Z)\ne 0$ or $\Hom_B(Z,Y)\ne 0$;

(c) if for every indecomposable $B$-module $Y$ we have $\dim_k \Hom_B(R_0,Y) \le 1$ then the poset $\Hom_B(R_0,{\mathcal C}')$ does not contain a hypercritical subposet, that is, a poset whose Hasse diagram is one of the following:

$$\xymatrix@R4pt@C2pt{
& & & & & & & & & & \bulito\ar[d] \\
{\mathcal N}_1=(1,1,1,1,1): & \bulito &\bulito &\bulito &\bulito &\bulito &
\hskip1truecm {\mathcal N}_2=(1,1,1,2): & \bulito &\bulito &\bulito &\bulito 
}$$
$$\xymatrix@R4pt@C2pt{
& & & & & & & & &  \bulito\ar[d] \\
& & & &\bulito\ar[d] & & & &\bulito \ar[d] &  \bulito\ar[d] \\
& &\bulito \ar[d] &\bulito\ar[d] &\bulito\ar[d] & & & &\bulito\ar[d] &  \bulito\ar[d] \\
&{\mathcal N}_3=(2,2,3): & \bulito &\bulito &\bulito &
\hskip1.5truecm {\mathcal N}_4=(1,3,4): & &\bulito &\bulito &\bulito 
}$$
$$\xymatrix@R4pt@C2pt{
& & & & & & & & &  \bulito\ar[d] \\
& & & & \bulito\ar[d]& & & & &  \bulito\ar[d] \\
& & & &\bulito\ar[d] & & & & &  \bulito\ar[d] \\
& & & &\bulito\ar[d] & & & & &  \bulito\ar[d] \\
& &\bulito \ar[d]\ar[rd] &\bulito\ar[d] &\bulito\ar[d] & & & &\bulito\ar[d] &  \bulito\ar[d] \\
&{\mathcal N}_5=(N,5): & \bulito &\bulito &\bulito &
\hskip1.5truecm {\mathcal N}_6=(1,2,6): & &\bulito &\bulito &\bulito 
}$$

\begin{propo}
With the notation above the following holds:

(i) ${\mathcal S}$ has at most three branching points and $t({\mathcal S}) \le 5$ terminal vertices. 

(ii) ${\mathcal S}'$ has at most two branching points and  $t({\mathcal S}') \le 4$. If $t({\mathcal S}')=4$, then $X$ is projective-injective (and $X$ is the unique sincere indecomposable module).

(iii) If $X$ is not projective-injective, then either $s(X)\ge 3$ or ${\mathcal S} \setminus \{ X^{\tau}\}$ has two non-linear components.

(iv) If $s(X)=2$ then $v$ has two exceptional vertices.

(v) If $s(X)=3$ then $v$ has a unique exceptional index.  

(vi) If $s(X)=4$ then $X$ is the unique indecomposable sincere $A$-module.
\end{propo}
\begin{proof}
(i) and (ii). Assume that ${\mathcal S}$ has $3$ or more branching points. Then the quiver ${\mathcal S}$ is not Euclidean and $R_0$ lies in ${\mathcal C}'$ which is a connecting component of $\Gamma_B$ accepting a slice ${\mathcal S}'$ with at least two branching points, then $t({\mathcal S}') \ge 4$. If $t({\mathcal S}) \ge 5$ then also $t({\mathcal S}') \ge 4$. Hence we shall assume that
$t({\mathcal S}') \ge 4$ which implies that in ${\mathcal C}'$ we get one of the following situations as subposets of $\Hom_B(R_0,{\mathcal S}')$:
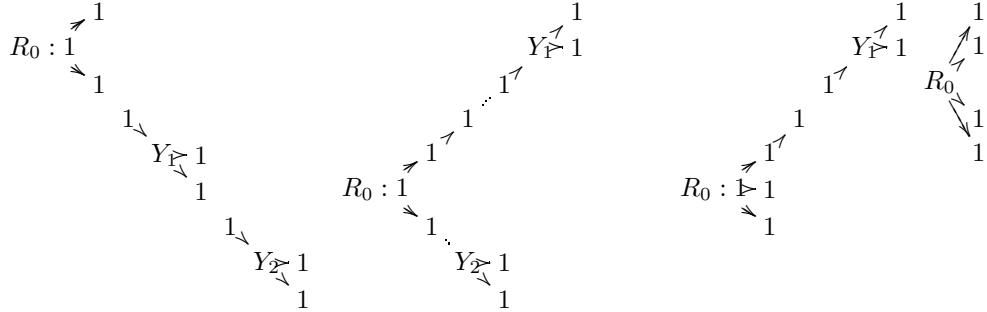
$$\footnotesize{\xymatrix@C.1pt@R.5pt{  
&1 & & & & & & &\\
R_0:1 \ar[ru] \ar[rd] & & & & & & & & \\
 & 1 \ar@{..}[rd] & & & & & & & \\
 & &    1 \ar[rd] & & &  & & & \\
 & & &  Y_1 \ar[rd]\ar[r]& 1 & & & &\\
 & & &  &  1 \ar@{..}[rd] & & & &\\
 & & &  &   & 1 \ar[rd] & & &\\
 & & &  &   &  &Y_2 \ar[rd]\ar[r] &1&\\
 & & &  &   &  & &1 &
}
\xymatrix@C.1pt@R.5pt{ 
 & & &  &  & 1 &\\
 & & &    &Y_1 \ar[ru]\ar[r] &1\\
 & & &   1 \ar[ru] & &\\
 & & 1 \ar@{..}[ru] & & & &\\
 &1\ar[ru] & & & & & & & & &\\
R_0:1 \ar[ru] \ar[rd] & & & & & & & & & &\\
 & 1 \ar@{..}[rd] & & & & & & & &\\
 & &   Y_2\ar[rd]\ar[r]& 1 & & & & \\
 & &    &1  & &  & & & &\\
}
\xymatrix@C.1pt@R.5pt{ 
 & & &  &  & 1 \\
 & & &    &Y_1 \ar[ru]\ar[r] &1\\
 & & &   1 \ar[ru] & &\\
 & & 1 \ar@{..}[ru] & & & \\
 &1\ar[ru] & & & & \\
R_0:1 \ar[ru] \ar[rd]\ar[r] & 1& & & &\\
 & 1 &  & & & 
}
\xymatrix@C.2pt@R.5pt{
&1 \\
&1 \\
R_0 \ar[ru]\ar[ruu]\ar[rd]\ar[rdd]&\\
&1\\
&1
}}$$

We conclude that $R_0$ is injective, since otherwise we get either an indecomposable module $Y$
with $\dim_k \Hom_B(R_0,Y) \ge 3$ (last case) or two non-comparable indecomposable modules $Y,Z$ with $\dim_k \Hom_B(R_0,Y)=2$. Then $P_s$ is projective-injective and $R_0$ is a sincere $B$-module. Since the poset $\Hom_B(R_0, {\mathcal S}')$ is tame, we get that ${\mathcal S}'$ is completely depicted in the above pictures. Hence there are at most $3$ branching point in ${\mathcal S}$ and at most $5$ terminal points. Moreover, precesily $t({\mathcal S}')=4$.

(iii) Assume that $s(X) \le 2$ and ${\mathcal S}\setminus \{ X \}$ has a linear component of type $\A_m$. We shall prove that $X$ is projective-injective by induction on $m$.

Let $X=X_{m,1} \to X_{m-1,2} \to \ldots \to X_{2,m-1}\to X_{1,m}$ be the chain of maps in ${\mathcal S}$ corresponding to the linear component in ${\mathcal S}\setminus \{ X \}$. Suppose first that none of the $X_{i,m-i+1}$ is injective. Then ${\mathcal C}$ contains the following modules and irreducible morphisms:  

$$\footnotesize
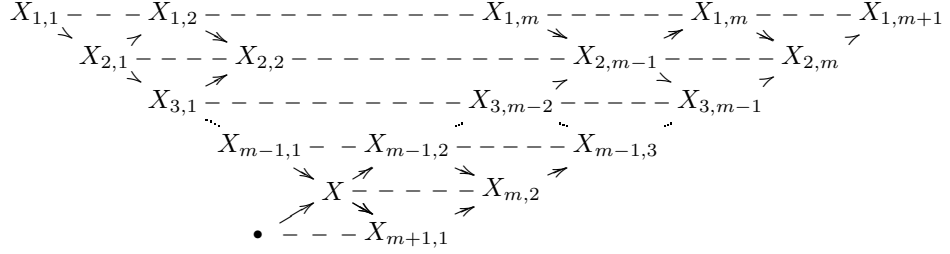
{\xymatrix@C1pt@R1pt{  
X_{1,1}\ar[rd]\ar@{--}[rr]&&X_{1,2}\ar[rd] \ar@{--}[rrrr] & & &&X_{1,m}\ar[rd]\ar@{--}[rr] & &X_{1,m}\ar[rd]\ar@{--}[rr]&&X_{1,m+1}\\
&X_{2,1} \ar[ru] \ar[rd] \ar@{--}[rr]& &X_{2,2}\ar@{--}[rrrr] & &  & &   X_{2,m-1}\ar[ru]\ar[rd]\ar@{--}[rr] & &X_{2,m}\ar[ru]\\
 && X_{3,1}\ar[ru] \ar@{..}[rd]\ar@{--}[rrrr]  & & & &  X_{3,m-2}\ar[ru]\ar@{..}[rd]\ar@{--}[rr]&&X_{3,m-1}\ar[ru]\\
 && &   X_{m-1,1} \ar[rd]\ar@{--}[rr]  & &  X_{m-1,2}\ar[rd]\ar@{..}[ru]\ar@{--}[rr]& &X_{m-1,3}\ar@{..}[ru]&\\
 && & &  X \ar[rd] \ar[ru]\ar[rd]\ar@{--}[rr]&  &X_{m,2}\ar[ru]&&\\
 && &\bulito\ar[ru]\ar@{--}[rr] &  &  X_{m+1,1}\ar[ru]& & & & &&\\
}}$$

Since every irreducible map is either mono or onto, a simple dimension argument yields that all descending maps $X_{i,j} \to X_{i-1,j}$ are mono and all ascending maps $X_{i,j} \to X_{i+1,j+1}$are onto. In particular, $X \to X_{m-1,2}$ is mono, contradicting the maximality of the root $v$.

Therefore one of $X_{i,m-i+1}$ is injective. Suppose that $i>1$, then $X_{i-1,m-i}$ is also injective since otherwise for the injective $I$ in the orbit of $X_{i-1,m-i}$ we have $\Hom_A(X,I)=0$, a contradiction to the sincerity of $X$. We get that all $X_{j,m-j+1}$, for $j \le i$ are injective.  In particular $S=X_{1,m}$ is simple injective and since $\dim_k S(a) + \dim_k S(b)= q_A(v, \udim S)= \dim_k \Hom_A(X,S)=1$ we may assume that $S=S_a$ is the simple associated to the index $a$. For the algebra $C=A/Ae_aA$ we get the one-point coextension$A=[X_{1,m}]C$ and therefore $C$ is tilted of type ${\mathcal S}\setminus \{ S_a \}$. The algebra $C$ accepts an omnipresent root $v-e_a$ which is realized by the module $X_{m-1,1}$. By induction hypothesis $X_{m-1,1}$ is projective-injective as a $C$-module and therefore $X$ is projective-injective as an $A$-module.

(iv) Assume that $X$ has two neighbors in ${\mathcal S}$ and $v= \udim X$ has only one exceptional index $a$. Since $v(a)=2$ then $X$ is not projective. Consider the Auslander-Reiten sequence $0 \to \tau X \to Y_1 \oplus Y_2 \to X \to 0$ where $Y_1,Y_2$ are indecomposable modules. Observe that for $i=1,2$,
$$\dim_k Y_i(a)= q_A(v,\udim Y_i)= \dim_k \Hom_A (Y_i,X)- \dim_k\Ext_A^1(X,Y_i)=1,$$
\noindent therefore 
$$2= v(a)=\dim_k \Hom_A(P_a,X)\le \dim_k Y_1(a) +\dim_k Y_2(a)$$ 
\noindent and 
$\dim_k \Hom_A (P_a, \tau X)=0,$ therefore
$$1= \dim_k P_a (a)=q_A(v, \udim P_a)= \dim_k \Hom_A(P_a,X)- \dim_k \Ext^1_A(X,P_a)=v(a)=2.$$
\noindent
A contradiction completing this case.

(v) The argument is similar to (iii) and is presented in \cite{22}.

(vi) If $s(X)=4$ then, as in (i), $X$ is injective, say with socle $S$. Let $X/S = \oplus_{i=1}^4 Y_i$, where the $Y_i$ are indecomposable not projective (since $X$ is sincere) $A$-modules. The dual argument of (i) yields that $X$ is projective. Then $X$ is the unique sincere $A$-module
\end{proof}

\subsection{Proof of Theorem 3} 
Let $A$ be an algebra whose Tits form $q_A:\Z^n \to \Z$ accepts a maximal omnipresent root $v$ with two exceptional indices $a$ and $b$. Then $A$ is a tame tilted algebra with a directing component ${\mathcal C}$ containing an indecomposable sincere module $X$ such that $\udim X=v$. We shall prove that $X$ is (up to isomorphism) the unique indecomposable sincere $A$-module. For this purpose we may assume that $X$ is not projective-injective.

Let $P_s$ be a last projective in ${\mathcal C}$ and $R_0$ its indecomposable radical as in the last paragraph. Hence $A=B[R_0]$ for a tilted strongly simply connected algebra $B$. Let ${\mathcal S}$ (resp. ${\mathcal S}'$) the slice in ${\mathcal C}$ (resp. ${\mathcal C}'$) whose unique source is $X$ (resp. $R_0$). Let $Y$ be an indecomposable sincere $A$-module non-isomorphic to $X$. We shall get a contradiction.

Assume that $a,b$ are exceptional indices of $v= \udim X$. By (4.3) we may suppose that $X$ has two neighbors in ${\mathcal S}$ (that is, $s(X)=2$) and $t({\mathcal S}')=3$. Then $B$ is tilted of type 
$$\footnotesize{\xymatrix{a_1 &a_2\ar[l]\ar@{--}[r]&a_{r-1}&a_r\ar[l]\ar[r]&a_{r+1}\ar@{--}[r]&a_{r+s} \ar[d]\ar[r]&b_1\ar@{--}[r]&b_{t-1}\ar[r]&b_t\\
                       & &&                    &&c_1\ar@{--}[d]&&& \\
                        &&&                    & &c_p & & & 
}}$$

\noindent where $a_r$ is the orbit of $R_0$. Let $Y_i$ (resp, $Z_j$, $Z'_r$) be the indecomposable modules corresponding to the vertex $a_i$ (resp. $b_j$, $c_r$) in ${\mathcal S}'$. First we shall observe that $r=2$: indeed, by (4.3), $r=1$ implies $X$ is projective-injective, a situation we have discarded. Hence $r\ge 2$. In case none of the $Y_i$ is injective, then $Y_1$ and $Z=\tau^{-1}Y_{r+s}$ are non-path comparable modules with 
$\dim_k\Hom_B(R_0,Z)=2$ and $\dim_k \Hom_B(R_0,Y_1)=1$ which implies that $A$ is a wild algebra. Therefore we get a module $Y_j$ which is injective, by the remarks at (4.3), and thus $X$ is projective-injective again. A contradiction showing that $r=2$.  

We show that one of the modules $Y_i$ is injective. Indeed, assume otherwise that none of the $Y_i$ is injective. In case $t>1$ or $p>1$, then as above, $A$ is wild. Therefore $t=1=p$ and then $A$ is tilted of Euclidean type ${\tilde \D}_{n-1}$, a situation which is also discarded. Suppose that $Y_j$ is injective. By \cite{31}, section 6.3, the injective $A$-module at the vertex $j$ is of the form ${\overline Y}_j =(Y_j,k, id)$ meaning that the restriction of the module to $B$ is $Y_j$, moreover ${\overline Y}_j= \tau^{-j}Y_i$ for $i=1$ or $2$, depending on the parity of $j$. We get the following structure for ${\mathcal C}$

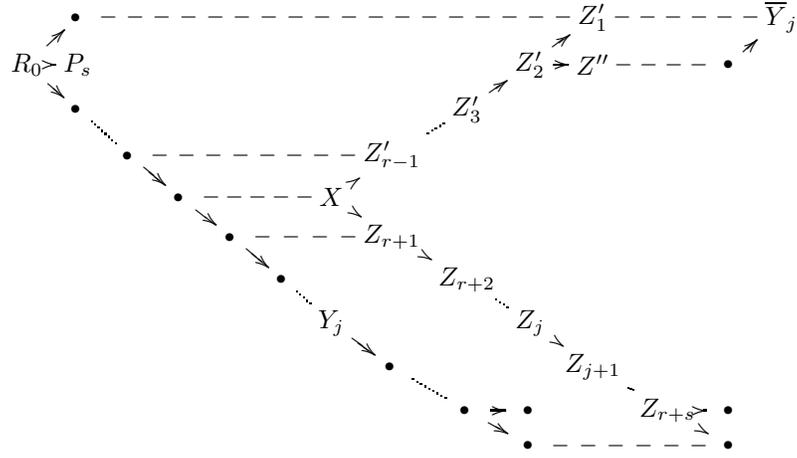
$$\footnotesize{\xymatrix@C1pt@R1pt{  
&\bulito \ar@{--}[rrrrrrrrr] & & & & & & & & &Z'_1\ar@{--}[rrr]&&& {\overline Y}_j\\
R_0 \ar[ru] \ar[rd] \ar[r]&P_s & & & & & & & &  Z'_2\ar[ru]\ar[r] & Z''\ar@{--}[rr]&&\bulito\ar[ru]\\
 & \bulito \ar@{..}[rd] & & & & & & & Z'_3\ar[ru]&\\
 & &    \bulito \ar[rd]\ar@{--}[rrrrr] & & &  & &  Z'_{r-1}\ar@{..}[ru]& &&\\
 & & &  \bulito \ar[rd]\ar@{--}[rrr]& & & X \ar[ru]\ar[rd]&  &&&\\
 & & &  &  \bulito \ar[rd] \ar@{--}[rrr]& & &Z_{r+1}\ar[rd] & &&\\
 & & &  &  &\bulito \ar@{..}[rd] & & &Z_{r+2} \ar@{..}[rd] &&\\
 & & &  &  & & Y_j  \ar[rd]& & &Z_j \ar[rd]&\\
 & & & & & & & \bulito\ar@{..}[rd]&&&Z_{j+1}\ar@{..}[rd]\\
 & & &  &  & & & & \bulito \ar[rd]\ar[r] &\bulito& &Z_{r+s}\ar[rd]\ar[r] & \bulito\\
 & & &  &  & &  && &\bulito \ar@{--}[rrr] & & &  \bulito
}}$$

We want to consider the position of $Y$ at the above picture. First calculate 
$$2 \le \dim_k Y(a) + \dim_k Y(b)= q_A(v,\udim Y) \le \dim_k \Hom_A(X,Y)+ \dim_k \Hom_A(Y,X). $$
Then we may suppose that $\Hom_A(X,Y) \ne 0$ and therefore $\Hom_A(Y,X)=0$. 

Observe that 
$$\dim_k Z_{r+1}(a)+ \dim_k Z_{r+1}(b)= q_A(v, \udim Z_{r+1})= \dim_k\Hom_A(X,Z_{r+1})=1,$$ 
and we may assume that $\dim_k Z_{r+1}(a)=1$ and $\dim_k Z_{r+1}(b)=0$. 
Similarly, we get $\dim_k Z_i(a)=1$ and $\dim_k Z_i(b)=0$, for all $i=r+1,\ldots,r+s$. 
Moreover, $\dim_k Z'_j(a)=0$ and $\dim_k Z'_j(b)=1$, for all $j=1,\ldots,r-1$ and $Z''(a)=0$. 
Since the function $\dim_k \Hom_A(P_a,?)$ is additive on Auslander-Reiten sequences, we get that $Z(a)=0$ at the modules of the form $Z=\tau^{-p}Y_i$ for $1 \le i \le r-1$ and $0 \le p \le i$. 

Since $Y$ is sincere, then $\dim_k \Hom_A(Y,{\overline Y}_j)\ne 0$ and $Y(a) \ne 0$. 
Therefore, 
$Y= \tau^{-p}Y_i$ for some $r \le i \le j$ and some $0 \le p \le i$. Another argument using the additivity of $\dim_k \Hom_A(P_a,?)$ implies that $Y=Z_i$ for some $r+1 \le i \le r+s$. But then $Y(b)=0$, contradicting the sincerity of $Y$. This completes the proof of Theorem 3.

\end{document}